\documentclass[12pt]{amsart}
\usepackage{amsmath,amssymb,amsbsy,amsfonts,latexsym,amsopn,amstext,cite,
                                               amsxtra,euscript,amscd,bm}
\usepackage{url}
\usepackage[colorlinks,linkcolor=blue,anchorcolor=blue,citecolor=blue,backref=page]{hyperref}
\usepackage{color}
\usepackage{graphics,epsfig}
\usepackage{graphicx}
\usepackage{float}

\usepackage{mathtools}
\usepackage{todonotes}

\hypersetup{breaklinks=true}

\RequirePackage{mathrsfs} \let\mathcal\mathscr

\usepackage[norefs,nocites]{refcheck}
 
\newfont{\teneufm}{eufm10}
\newfont{\seveneufm}{eufm7}
\newfont{\fiveeufm}{eufm5}
\newfam\eufmfam
                \textfont\eufmfam=\teneufm \scriptfont\eufmfam=\seveneufm
                \scriptscriptfont\eufmfam=\fiveeufm
\def\bbbc{{\mathchoice {\setbox0=\hbox{$\displaystyle\rm C$}\hbox{\hbox
to0pt{\kern0.4\wd0\vrule height0.9\ht0\hss}\box0}}
{\setbox0=\hbox{$\textstyle\rm C$}\hbox{\hbox
to0pt{\kern0.4\wd0\vrule height0.9\ht0\hss}\box0}}
{\setbox0=\hbox{$\scriptstyle\rm C$}\hbox{\hbox
to0pt{\kern0.4\wd0\vrule height0.9\ht0\hss}\box0}}
{\setbox0=\hbox{$\scriptscriptstyle\rm C$}\hbox{\hbox
to0pt{\kern0.4\wd0\vrule height0.9\ht0\hss}\box0}}}}
\def\bbbq{{\mathchoice {\setbox0=\hbox{$\displaystyle\rm
Q$}\hbox{\raise 0.15\ht0\hbox to0pt{\kern0.4\wd0\vrule
height0.8\ht0\hss}\box0}} {\setbox0=\hbox{$\textstyle\rm
Q$}\hbox{\raise 0.15\ht0\hbox to0pt{\kern0.4\wd0\vrule
height0.8\ht0\hss}\box0}} {\setbox0=\hbox{$\scriptstyle\rm
Q$}\hbox{\raise 0.15\ht0\hbox to0pt{\kern0.4\wd0\vrule
height0.7\ht0\hss}\box0}} {\setbox0=\hbox{$\scriptscriptstyle\rm
Q$}\hbox{\raise 0.15\ht0\hbox to0pt{\kern0.4\wd0\vrule
height0.7\ht0\hss}\box0}}}}
\def\bbbt{{\mathchoice {\setbox0=\hbox{$\displaystyle\rm
T$}\hbox{\hbox to0pt{\kern0.3\wd0\vrule height0.9\ht0\hss}\box0}}
{\setbox0=\hbox{$\textstyle\rm T$}\hbox{\hbox
to0pt{\kern0.3\wd0\vrule height0.9\ht0\hss}\box0}}
{\setbox0=\hbox{$\scriptstyle\rm T$}\hbox{\hbox
to0pt{\kern0.3\wd0\vrule height0.9\ht0\hss}\box0}}
{\setbox0=\hbox{$\scriptscriptstyle\rm T$}\hbox{\hbox
to0pt{\kern0.3\wd0\vrule height0.9\ht0\hss}\box0}}}}
\def\bbbs{{\mathchoice
{\setbox0=\hbox{$\displaystyle     \rm S$}\hbox{\raise0.5\ht0\hbox
to0pt{\kern0.35\wd0\vrule height0.45\ht0\hss}\hbox
to0pt{\kern0.55\wd0\vrule height0.5\ht0\hss}\box0}}
{\setbox0=\hbox{$\textstyle        \rm S$}\hbox{\raise0.5\ht0\hbox
to0pt{\kern0.35\wd0\vrule height0.45\ht0\hss}\hbox
to0pt{\kern0.55\wd0\vrule height0.5\ht0\hss}\box0}}
{\setbox0=\hbox{$\scriptstyle      \rm S$}\hbox{\raise0.5\ht0\hbox
to0pt{\kern0.35\wd0\vrule height0.45\ht0\hss}\raise0.05\ht0\hbox
to0pt{\kern0.5\wd0\vrule height0.45\ht0\hss}\box0}}
{\setbox0=\hbox{$\scriptscriptstyle\rm S$}\hbox{\raise0.5\ht0\hbox
to0pt{\kern0.4\wd0\vrule height0.45\ht0\hss}\raise0.05\ht0\hbox
to0pt{\kern0.55\wd0\vrule height0.45\ht0\hss}\box0}}}}
\def\bbbz{{\mathchoice {\hbox{$\sf\textstyle Z\kern-0.4em Z$}}
{\hbox{$\sf\textstyle Z\kern-0.4em Z$}} {\hbox{$\sf\scriptstyle
Z\kern-0.3em Z$}} {\hbox{$\sf\scriptscriptstyle Z\kern-0.2em
Z$}}}}

\newtheorem{theorem}{Theorem}
\newtheorem{lemma}[theorem]{Lemma}

\newtheorem{cor}[theorem]{Corollary}

 \numberwithin{equation}{section}
  \numberwithin{theorem}{section}

\def\squareforqed{\hbox{\rlap{$\sqcap$}$\sqcup$}}
\def\qed{\ifmmode\squareforqed\else{\unskip\nobreak\hfil
\penalty50\hskip1em\null\nobreak\hfil\squareforqed
\parfillskip=0pt\finalhyphendemerits=0\endgraf}\fi}

\def\cM{{\mathcal M}}

\def\cO{{\mathcal O}}

\def \sf {\mathfrak s}

\newcommand{\ignore}[1]{}

\def \C{\mathbb{C}}

\def \Z{\mathbb{Z}}

\def \R{\mathbb{R}}
\def \Q{\mathbb{Q}}
\def \N{\mathbb{N}}

\def\mand{\qquad\mbox{and}\qquad}

\def\\{\cr}
\def\({\left(}
\def\){\right)}
\def\fl#1{\left\lfloor#1\right\rfloor}

\def\eps{\varepsilon}

\def\fB{{\mathfrak B}}

\newcommand\wH{{\rm H}}
\newcommand\vb{{\bf v}}

\newcommand\xb{{\bf x}}

\newcommand\zb{{\bf z}}

\def\bx{\mathbf x}
\def\bu{\mathbf u}

\newcommand\p{\mathfrak{p}}

\begin{document}

\title[Multiplicatively dependent vectors]{On the distribution of multiplicatively dependent vectors}

\author[S. Konyagin]{Sergei V.~Konyagin}
\address{Steklov Mathematical Institute,
8, Gubkin Street, Moscow, 119991, Russia} 
\email{konyagin@mi-ras.ru}

\author[M. Sha]{Min Sha}
\address{School of Mathematical Sciences, South China Normal University, Guangzhou, 510631, China}
\email{min.sha@m.scnu.edu.cn}

\author[I. E. Shparlinski]{Igor E. Shparlinski}
\address{Department of Pure Mathematics, University of New South Wales,
 Sydney, NSW 2052, Australia}
\email{igor.shparlinski@unsw.edu.au}

\author[C. L. Stewart]{Cameron L. Stewart}
\address{Department of Pure Mathematics, University of Waterloo, 
Waterloo, Ontario, N2L 3G1, Canada}
\email{cstewart@uwaterloo.ca}

\begin{abstract} 
In this paper, we study the distribution of multiplicatively dependent vectors. 
For example, although they have zero Lebesgue measure, they are everywhere dense both in $\R^n$ and $\C^n$. 
We also study this property in a more detailed manner by considering the covering radius of such vectors. 
\end{abstract}


\maketitle

\section{Introduction}

 \subsection{Background }
Let $n \ge 2$ be a positive integer, $R$ be a ring with identity and let $\vb=(v_1,\dots,v_n)$ be in $R^n$. 
We say that the vector $\vb$ is multiplicatively dependent if all its coordinates are non-zero and 
there is a non-zero integer vector $\mathbf{k}=(k_1,\dots,k_n)$ in $\mathbb{Z}^n$ for which
\begin{equation} 
\label{eq:0}
v^{k_1}_1\cdots v^{k_n}_n=1.
\end{equation} 
Let $S$ be a subset of $R$. We denote by $\cM_n(S)$ the set of multiplicatively dependent vectors with coordinates in $S$. 

In 2018 Pappalardi, Sha, Shparlinski and Stewart~\cite{PSSS} gave asymptotic estimates 
for the number of multiplicatively dependent vectors whose coordinates are  algebraic numbers 
of bounded height and of fixed degree or within a fixed number field. 
For example, it follows from~\cite[Equation~(1.16)]{PSSS} that 
for any integer $n \ge 2$ there is a positive number $c_0(n)$ such that 
the number of elements of  $\cM_n(\Z)$ whose coordinates are at most $H$ in absolute value is
\begin{equation} 
\label{eq:count H_Z}
n(n+1)(2H)^{n-1} + O\(H^{n-2} \exp(c_0(n)\log H/\log \log H)\).
\end{equation} 

The multiplicative dependence of algebraic numbers has also been  studied from other aspects. 
These include bounding the heights of multiplicatively dependent algebraic numbers (see~\cite{Stew}),  
studying points on an algebraic curve whose coordinates are non-zero algebraic numbers and multiplicatively dependent (see~\cite{BCMOS, BS, BMZ, OSSZ1}), 
 investigating multiplicative dependence of rational values 
(see~\cite{DS2016, OSSZ2}), considering multiplicative dependence among iterated values of rational functions (see~\cite{BOSS, OSSZ2}), 
and studying multiplicative dependence modulo groups (see~\cite{BBMOS, BOSS}).

In this paper, we  study the distribution of the elements of $\cM_n(S)$ 
when $S$ is a subset of the real numbers $\R$ or the complex numbers $\C$ with number theoretic interest. 
Note that the sets  $\cM_n(\R)$ and  $\cM_n(\C)$ 
have zero Lebesgue measure, since they are countable unions of hypersurfaces and each hypersurface in $\R^n$ or $\C^n$ has zero Lebesgue measure. 
On the other hand, our results imply that $\cM_n(\R)$ and $\cM_n(\C)$ are dense in $\R^n$ and $\C^n$ respectively; 
see Theorem~\ref{thm:Dense_S_R} and Theorem~\ref{thm:Dense_S_C}.

Let $K$ be a number field, which we always identify with one of its models, 
that is, $K=\Q(\alpha)$ for some algebraic number $\alpha$. Recall, that alternatively, one can 
think of $K$ as
$\Q[X]/f(X)\Q[X]$ for an irreducible polynomial $f(X) \in \Z[X]$ and then consider its various embeddings in $\C$ 
and $\R$.

As usual, we define the degree of $K$ to be the degree $[K:\Q]$ of the field extension $K/\Q$.  
Let $\cO_K$ denote the ring of integers of $K$. 
We study the distribution of $\cM_n(K)$ and $\cM_n(\cO_K)$ in $\R^n$ and also in $\C^n$. 
Among other results, we prove that $\cM_n(K\cap \R)$ is dense in $\R^n$, and $\cM_n(\cO_K \cap \R)$ is dense in $\R^n$ if $\cO_K \cap \R \ne \Z$. 
Further, $\cM_n(K)$ is dense in $\C^n$ if $K \subsetneq \R$, 
and $\cM_n(\cO_K)$ is dense in $\C^n$ if $K \subsetneq \R$ and $[K:\Q] \ge 3$. 
Then, to study the cases of $\cM_n(\Z)$, which is not dense in $\R^n$, and of $\cM_n(\cO_K)$ when $K$ is an imaginary quadratic field, 
which is not dense in $\C^n$, we introduce a refinement of the notion of the covering radius of a set 
and use it to show that there are significant irregularities in the distribution of the elements of $\cM_n(\Z)$ in $\R^n$ and of $\cM_n(\cO_K)$ in $\C^n$.

 \subsection{Density results for multiplicatively dependent vectors}

We say that a subset $S$ of a ring $R$ is \textit{closed under powering} if for any $\alpha$ in $S$ we also have 
$\alpha^m$ in $S$ for every non-zero integer $m$.

\begin{theorem}
\label{thm:Dense_S_R} 
Let  $n \ge 2$ and let $S$ be a dense subset of $\R$
which is closed under powering.
Then $\cM_n(S)$ is dense in $\R^n$.
\end{theorem}

We remark that if $S$ is a dense subset of $\R$ which is not closed under powering, then  $\cM_n(S)$ may not be dense in $\R^n$.
For example, let $S$ be the set of all rational numbers of the form $p/q$ or $-p/q$ with distinct primes   $p, q$. 
Then by~\cite[Theorem~4]{HS} $S$ is dense in $\R$, but  $\cM_n(S)$ is not dense in $\R^n$ for any $n \ge 2$ 
(see Section~\ref{sec:rem} for more details). 

 Since the rationals are dense in $\R$ and closed under powering, we deduce the following result. 

\begin{cor}
\label{cor:Dense_Q_R} 
Let $n  \ge 2$.  
Then   $\cM_n(\Q)$ is dense in $\R^n$. 
\end{cor}

Let $K$ be a number field of degree at least $2$. Plainly $\cM_n(K \cap \R)$ is dense in $\R^n$ by Corollary~\ref{cor:Dense_Q_R} 
since $\Q$ is contained in $K\cap \R$. 
Furthermore, if $\cO_K \cap \R \ne \Z$, 
then $\cO_K \cap \R$ is easily seen to be dense in $\R$, and since it is closed under powering we have the following result.

\begin{cor}
\label{cor:Dense_OK_R} 
Let $n \ge 2$, and let $K$ be a number field.  
If $\cO_K \cap \R \ne \Z$,  then   $\cM_n(\cO_K \cap \R)$ is dense in $\R^n$. 
\end{cor}

We next establish the analogue of Theorem~\ref{thm:Dense_S_R} when $\R$ is replaced by $\C$. 

\begin{theorem}
\label{thm:Dense_S_C} 
Let $n  \ge 2$ and let $S$ be a dense subset of  $\C$
which is closed under powering.
Then $\cM_n(S)$ is dense in $\C^n$.
\end{theorem}

As before, we remark that in Theorem~\ref{thm:Dense_S_C}
 the condition that  $S$ be closed
under powering cannot be removed. 
For example, let $S$ be the set of all  algebraic numbers of the form $\zeta p/q$ with $\zeta$ a root of unity and with $p$ and $q$ distinct primes.  
Then $S$ is dense in $\C$, but  $\cM_n(S)$ is not dense in $\C^n$ for any $n \ge 2$
(see Section~\ref{sec:rem}). 

If $K$ is a number field  not  contained in $\R$, then $K$ is dense in $\C$ and 
we deduce our next result.

 \begin{cor}
\label{cor:Dense_K_C} 
Let $n \ge 2$, and let $K$ be a number field. 
If $K$ is  not  contained in $\R$, 
then   $\cM_n\(K\)$ is dense in $\C^n$. 
\end{cor}

Further, by Lemma~\ref{lem: Dense ANF} below, if $K$ is a number field of degree at 
least $3$ which is not  contained in $\R$, then $\cO_K$ is dense in $\C$ and we have the following 
result.
 \begin{cor}
\label{cor:Dense_OK_C} 
Let $n \ge 2$,  and let $K$ be a number field. 
If $[K:\Q] \ge 3$ and $K$ is not  contained in $\R$, 
then   $\cM_n\(\cO_K\)$ is dense in $\C^n$. 
\end{cor}

Clearly, one can see that the converses of Corollaries~\ref{cor:Dense_OK_R}, \ref{cor:Dense_K_C} 
and~\ref{cor:Dense_OK_C} are true.

 \subsection{Covering radius  of the set of multiplicatively dependent vectors}
 
Let $S$ be a subset of $\R$. 
The \textit{covering radius} of  $\cM_n(S)$ in $\R^n$ is defined as 
$$
\rho_n(S) = \sup_{\xb \in \R^n}~
 \inf_{\vb \in \cM_n(S)} \|\xb - \vb\|,
$$
where $\|\xb\|$ is the Euclidean norm of $\xb=(x_1,\ldots,x_n)\in \R^n$, that is, 
$$
\|\xb\| = \sqrt{x_1^2 + \ldots + x_n^2} . 
$$
 Clearly, $\cM_n(S)$ is dense in $\R^n$ if and only if $\rho_n(S)=0$.  
Let $K$ be a number field. 
Then, for any integer $n \ge 2$ it follows from Corollary~\ref{cor:Dense_Q_R} that $\rho_n(K \cap \R)=0$ 
and from Corollary~\ref{cor:Dense_OK_R} that $\rho_n(\cO_K \cap \R)=0$ provided that $\cO_K \cap \R \ne \Z$. 
On the other hand,
trivially $\rho_n(\Z) \ge 1$ and it follows from~\eqref{eq:count H_Z} that
 in fact $\rho_n(\Z) = \infty$; see~\eqref{eq:H}. 
In this case we introduce a finer measure in order to study more precisely the distribution of multiplicatively dependent vectors with integer coordinates.
For $H>1$ we define 
$$
\rho_n(H;\Z) = \sup_{\substack{\xb \in \R^n\\ \|\xb\| \le H}}~
  \inf_{\vb \in \cM_n(\Z)} \|\xb - \vb\|. 
$$

Each point of $\cM_n(\Z)$ which is in the ball of radius $H$ centered at the origin has coordinates which are at most $H$ in absolute value. 
By~\eqref{eq:count H_Z} there is a positive number $c_1(n)$, which depends on $n$, such that the number of such points is at most $c_1(n)H^{n-1}$. 

In addition there is a positive number $c_2(n)$, which depends on $n$, such that the volume of a ball of radius $r$ in $\R^n$ is $c_2(n)r^n$. Thus, the ball of radius $H$ centered at the origin has volume $c_2(n)H^{n}$, 
and so in order to cover it with balls of radius $r$ centered at the points of $\cM_n(\Z)$ 
which lie in it we must have $c_1(n)r^nH^{n-1}$ larger than $H^n$. In particular we must have
\begin{equation} 
\label{eq:H}
\rho_n(H;\Z) \geq c_3(n)H^{1/n},
\end{equation}
where $c_3(n)= c_1(n)^{-1/n}$. 

If the points of $\cM_n(\Z)$ were evenly distributed, then the lower bo\-und~\eqref{eq:H} would be sharp. 
However, the distribution of the points is in fact remarkably non-uniform. 
Certainly there are many points which are close to each other in $\cM_n(\Z)$,  
since if $n>2$ then $(2^k,2,x_3, \ldots, x_n)$ is in $\cM_n(\Z)$ for each positive integer $k$  
whenever $x_3, \ldots, x_n$ are non-zero integers. 
Furthermore for each positive integer $k$ both $(2^{k},2)$ and $(2^{k},4)$ are in $\cM_2(\Z)$. 
In addition there are large regions of $\R^n$ devoid of points of $\cM_n(\Z)$. 

In the sequel, the implied constants in the symbols $O$ and
$\ll$  may depend on $n$ and $K$.
(We   recall that  $U=O(V)$ and $U\ll V$ are  equivalent to the inequality $|U|\le cV$ with some positive number $c$.)

 In particular we prove the following result, which shows the true order of magnitude of $\rho_n(H;\Z)$ 
 to be spectacularly different from that suggested by~\eqref{eq:H}. 

\begin{theorem}
\label{thm:rho H Z} 
 For $H > 1$, we have 
$$
H \ll \rho_2(H;\Z) \ll H, 
$$
and for $n\ge 3$ 
$$
H / (\log H)^{C_0(n)}  \ll \rho_n(H;\Z) \ll  H \frac {(\log\log H)^{n-1}}{(\log H)^{n-2}}, 
$$
where $C_0(n)$ is a positive number which is effectively computable in terms of $n$.  
\end{theorem}

The  lower bound for  $\rho_n(H;\Z)$ with  $n \ge 3$ in  Theorem~\ref{thm:rho H Z}
is established by means of a  result of Tijdeman~\cite{Tij1} on gaps between integers composed 
of a fixed set of primes.  For the upper bound we use an explicit construction. 

Similarly, if $T$ is a subset of $\C$,  
then   
the \textit{covering radius} of  $\cM_n(T)$ in $\C^n$ is defined as 
$$
\mu_n(T) = \sup_{\zb \in \C^n}~
  \inf_{\vb \in \cM_n(T)} \|\zb - \vb\|, 
$$
where $\|\zb\|$ is the Euclidean norm of $\zb=(z_1,\ldots,z_n) \in \C^n$,  that is, 
$$
\|\zb\| = \sqrt{|z_1|^2 + \ldots + |z_n|^2}. 
$$

Clearly, for any subset $T$ of $\C$, $\cM_n(T)$ is dense in $\C^n$ if and only if $\mu_n(T)=0$. 
By Corollaries~\ref{cor:Dense_K_C} and~\ref{cor:Dense_OK_C} it remains to determine 
 $\mu_n(\cO_K)$ 
 for $n \ge 2$ when $K$ is an imaginary quadratic field.  
 By~\cite[Equation~(1.7)]{PSSS} the number of elements of $\cM_n(\cO_K)$ whose coordinates have absolute Weil height, see~\eqref{eq: Weil} below,  at most $H$ is
$$
\frac{n(n+1)}{2}w\(\frac{2\pi H^2}{|D|^{1/2}}\)^{n-1} + O\(H^{2n-3}\), 
$$
where $w$ denotes the number of roots of unity in $K$ and $D$ denotes the discriminant of $K$. 
It follows, as in~\eqref{eq:H},  that in this case  $\mu_n(\cO_K) = \infty$; see also
 the lower bounds of Theorem~\ref{thm:ImQuad}.
 As in the real case,  we introduce the following more refined concept. 
For $H>1$ and $K$ an  imaginary quadratic field,  we put 
$$
\mu_n(H;\cO_K) = \sup_{\substack{\xb \in \C^n\\ \|\xb\| \le H}}~
  \inf_{\vb \in \cM_n(\cO_K)} \|\xb - \vb\|. 
$$

 \begin{theorem}
\label{thm:ImQuad}  
Let $K$ be an imaginary quadratic field, and let $H$ be a real number with $H > 1$.
Then, there exists a number $C_0(n)$, which is effectively computable  in terms of $n$, such that
$$
H \ll \mu_2(H;\cO_K) \ll H, 
$$
and for $n \ge 3$,
$$ 
H / (\log H)^{C_0(n)} \ll \mu_n(H;\cO_K) \ll   H \frac {\log\log H}{(\log H)^{1/2}}. 
$$
\end{theorem}

For the proof of the lower bound in Theorem~\ref{thm:ImQuad} we again appeal to the result of Tijdeman~\cite{Tij1} while for the upper bound we give an explicit construction. 

We note that an alternative approach to upper bounds in Theorems~\ref{thm:rho H Z} and~\ref{thm:ImQuad} can be given using the results of Tijdeman~\cite{Tij2} on gaps between products 
of powers of fixed primes and their analogue for algebraic 
numbers due to Stewart~\cite{Stewart} (see also~\cite{Stewart1}). However this approach leads to quantitatively weaker 
bounds.

\section{Preliminaries}

\subsection{Density of algebraic integers in $\C$}

We believe that the main result of this section is of 
independent interest. It is also needed for the proof of Corollary~\ref{cor:Dense_OK_C}.

\begin{lemma}  \label{lem:Sab}
Let $\alpha$ and $\beta$ be complex numbers which are not in $\R$ with $1, \alpha$ and $\beta$ linearly independent over $\Q$ and for which 
$$\Q(\alpha,\beta)\cap \R = \Q.$$ 
Then, the set 
$$
S_{\alpha,\beta} = \{a+b\alpha+c\beta:\, a,b,c \in \Z\}
$$
is dense in $\C$. 
\end{lemma}

\begin{proof}
Let $\varepsilon$ be a real number with $0 < \varepsilon <1$,  and let $x+yi$ be in $\C$
with $x,y \in \R$. 
We want to show that there are elements of $S_{\alpha, \beta}$ within 
$\varepsilon$ of  $x+yi$. 
Without loss of generality, we can assume that 
\begin{equation} 
\label{eq:xy}
1 \leq x< 2 \mand y \ge 0.
\end{equation} 

Let $K=\Q(\alpha,\beta)$. 
Note that $1, \alpha, \beta$ are linearly independent over $\Q$. 
Then, for any integer $n$, the numbers 
 $1, \alpha+n, \beta$ are also linearly independent over $\Q$.
So we can assume that 
$$
\alpha = a + bi, 
$$ 
where $i=\sqrt{-1}$ is the imaginary unit and $a, b$ are positive real numbers. 

Since $\alpha$ is not a real number, $\C = \R(\alpha)$ and so there exist real numbers $r$ and $s$ with 
\begin{equation} 
\label{eq: gamma rs}
\beta = r+s\alpha.
\end{equation} 
We cannot have both $r$ and $s$ in $\Q$, since $1, \alpha, \beta$ are linearly independent over $\Q$. 
Moreover, neither  $r$ nor $s$ is in $\Q$. Indeed, if $r$ is in $\Q$, then 
$s = (\beta - r)/\alpha$ is in $K\cap \R$, and hence by our assumption $s$ is in $\Q$, which is a contradiction. A similar 
argument also applies if $s$ is in $\Q$.

Suppose that $1, r, s$ are  linearly dependent over $\Q$. 
Then, there exist integers  $j,k$ and $ \ell$, not all zero, such that 
\begin{equation} 
\label{eq: jkl rs}
j + kr +\ell s = 0. 
\end{equation} 
Since $r$ and $s$ are irrational, we have  $k \ell \ne 0$. 
By~\eqref{eq: gamma rs} and~\eqref{eq: jkl rs}, 
$$
j + k(\beta - s\alpha) +\ell s = 0
$$
or
$$
j + k\beta  =  (k \alpha - \ell) s .
$$
Since $\beta$ is non-real, $j + k\beta$ is non-zero, and so is $k \alpha - \ell$. Then,  
$$
s= \frac{j + k\beta}{k \alpha - \ell},
$$
and we see that $s \in K$. Since $s$ is real and $K\cap \R = \Q$, we  deduce that $s$ is in $\Q$, which gives a contradiction.

Therefore, we  must have that $1, r, s$ are  linearly independent over $\Q$. 

For any real number $x$, let $\fl{x}$ denote the integer part of $x$ (that is, the largest integer not greater than $x$), 
and let $\{x\}= x - \fl{x}$ denote the fractional part of $x$.
By Kronecker's Theorem (see~\cite[Theorem~443]{Hardy}), there exists a positive integer $q$ such that 
\begin{equation} 
\label{eq: qrs}
\frac{\varepsilon}{4} < \{q r\} < \frac{\varepsilon}{2}
\mand 
\{q s\} < \frac{\varepsilon^2}{20 \max\{a,b\}}. 
\end{equation} 
Put 
$$
\gamma_1 = q \beta - \fl{qr} - \fl{qs}\alpha, 
$$
and note that $\gamma_1 \in S_{\alpha,\beta}$. Further, by~\eqref{eq: gamma rs}
$$
\gamma_1 = \{qr\} + \{qs\}\alpha, 
$$
and since $\alpha = a+bi$, we have
\begin{equation} 
\label{eq: gamma1 lambda}
\gamma_1 =\lambda + \{qs\}bi, 
\end{equation} 
where 
\begin{equation} 
\label{eq: lambda}
\lambda = \{qr\} +  \{qs\}a.
\end{equation} 
We now define $q_1$ to be the integer  for which 
\begin{equation} 
\label{eq: q1}
q_1  \{qs\}b \le y <(q_1 +1) \{qs\}b, 
\end{equation} 
which is possible since $s$ is irrational and thus $ \{qs\} > 0$.
Then
$$
q_1 \gamma_1 = q_1 \lambda  + q_1  \{qs\}b i.
$$
We put 
$$
\gamma_2 = q_1 \gamma_1 -  \fl{q_1\lambda} .
$$
Note that $\gamma_2 \in S_{\alpha,\beta}$ and 
\begin{equation} 
\label{eq: gamma2 q1}
\gamma_2 = \{q_1\lambda\} + q_1  \{qs\}b i.
\end{equation} 
We now choose $q_2$ to be the integer  for which 
\begin{equation} 
\label{eq: q2}
\{q_1\lambda\} + q_2 \lambda \le x < \{q_1\lambda\} +  (q_2 +1) \lambda, 
\end{equation} 
which is possible, since the irrationality of $r,s$ and the positivity of $a$, 
together with~\eqref{eq: lambda}, 
imply  that $\lambda$ is positive.
Thus, by~\eqref{eq:xy} and~\eqref{eq: q2} we have 
$$
q_2 \lambda \le x < 2, 
$$
and by~\eqref{eq: qrs} and~\eqref{eq: lambda} we have 
$$
\lambda \ge  \{qr\} > \frac{\varepsilon}{4}.
$$
Therefore 
\begin{equation} 
\label{eq: q2 small}
  q_2<  \frac{8} {\varepsilon}.  
\end{equation}
Note that $q_1$ and $q_2$ are non-negative.

Observe that $\gamma_2 + q_2 \gamma_1$ is in  $S_{\alpha,\beta}$,
and by~\eqref{eq: gamma1 lambda} and~\eqref{eq: gamma2 q1}, 
\begin{align*}
\left| \gamma_2 + q_2 \gamma_1 - (x+y i)\right| &
 \leq \left| \{q_1\lambda\} +   q_2 \lambda - x\right| + \left| q_1  \{qs\}b + q_2  \{qs\} b -y \right| . 
\end{align*}
Thus, by~\eqref{eq: q1} and~\eqref{eq: q2} we have 
\begin{equation} 
\label{eq: est}
  \left| \gamma_2 + q_2 \gamma_1 - (x+y i)\right| \leq \lambda + (q_2+1)\{qs\}b.
\end{equation}
By~\eqref{eq: qrs} and ~\eqref{eq: lambda}
\begin{equation} 
\label{eq: est1}
\lambda \leq \frac{\varepsilon}{2} + \frac{\varepsilon^2}{20},
\end{equation}
and by~\eqref{eq: qrs} and~\eqref{eq: q2 small}
\begin{equation} 
\label{eq: est2}
 (q_2 +1)  \{qs\}b \leq \( \frac{8} {\varepsilon}+1\)  \frac{\varepsilon^2}{20} = \frac{2}{5}\varepsilon + \frac{\varepsilon^2}{20}.
\end{equation}

Thus, by~\eqref{eq: est},~\eqref{eq: est1} and~\eqref{eq: est2},   
$$
\left| \gamma_2 + q_2 \gamma_1 - (x+y i)\right|
<\varepsilon 
$$
as required. 
\end{proof}

Note that the set $S_{\alpha,\beta}$ in Lemma~\ref{lem:Sab} 
is in fact the sum of two lattices $\Z + \Z \alpha$ and $\Z + \Z \beta$. 
Although each lattice is not dense in the plane, 
the sum of the two lattices is dense in the plane under the condition in Lemma~\ref{lem:Sab}. 

We remark that the condition $\Q(\alpha,\beta) \cap \R = \Q$ in Lemma~\ref{lem:Sab} 
cannot be removed. For example, choosing $\alpha=i, \beta = \sqrt{2}+i$, we have that 
$\Q(\alpha,\beta) \cap \R = \Q(\sqrt{2})$ and $S_{\alpha,\beta}$ is not dense in $\C$. 

 \begin{lemma}
\label{lem: Dense ANF} 
Let   $K$ be a  number field.
Then, the ring of integers $\cO_K$ is dense in $\C$ if and 
only if $K$  is not  contained in $\R$ and $[K:\Q] \ge 3$. 
\end{lemma}

\begin{proof} 
Clearly the condition  $[K:\Q]\ge 3$ is necessary. 
Indeed, if $[K:\Q]=2$ and $K$ is not contained in $\R$, then $\cO_K$ forms a lattice  in the plane 
and so cannot be dense in $\C$. 

Let us now  suppose that $K$  is not  contained in $\R$ and that  $[K:\Q]\ge 3$. 
We consider the following two cases. 

\subsubsection*{Case~1.} 
We first consider the case when $K \cap \R= \Q$. Since $K \cap \R= \Q$ and $[K:\Q]\ge 3$, 
there exist non-real algebraic integers $\alpha$ and $\beta$ in $\cO_K$ 
such that $1, \alpha, \beta$ are linearly independent over $\Q$. 
By Lemma~\ref{lem:Sab}, this case is done. 

\subsubsection*{Case~2.} 
We now consider the case when $K \cap \R\ne \Q$ or equivalently when $\cO_K  \cap \R\ne \Z$. 
Then, there exists a real algebraic integer  $\alpha \in \cO_K$ which is not in $\Z$ and so is irrational. 
Further, since $K$ is not contained in $\R$, 
there exists an element $\beta \in \cO_K$ with $\beta = a + bi$ where $a$ and $b$ are real numbers with $b > 0$. 

Let $\varepsilon$ be a real number with $0 < \varepsilon <1$,  and let $x+yi$ be in $\C$
with $x, y \in \R$.  
We now show that there are elements of $\cO_K$ within 
$\varepsilon$ of  $x+yi$. 

Since $\alpha$ is irrational,   
we can choose integers $c$ and $d$ such that 
$$
|c + d\alpha - y/b| < \frac{\varepsilon}{2b}.
$$
Similarly, we can choose integers $r$ and $s$ with 
$$
|r+s\alpha +ac + ad\alpha - x| < \frac{\varepsilon}{2}.
$$
We then put $\lambda = r+s\alpha +(c + d\alpha) \beta$. Observe that $\lambda \in \cO_K$ 
and 
$$
|\lambda - (x+ yi)| < \frac{\varepsilon}{2} +  \frac{\varepsilon}{2} =  \varepsilon,  
$$
as required. 
\end{proof}  

Combining Theorem~\ref{thm:Dense_S_C} with Lemma~\ref{lem: Dense ANF}, 
we obtain Corollary~\ref{cor:Dense_OK_C}.

\subsection{Multiplicative dependence of algebraic numbers}

For any algebraic number $\alpha$ of degree $d\ge 1$, let
$$
f(x)=a_dx^d+\cdots+a_1x+a_0
$$
be the minimal polynomial of $\alpha$ over the integers $\Z$  (so with content $1$ and positive leading coefficient). Suppose that $f$ factors as
$$
f(x)=a_d(x-\alpha_1)\cdots (x-\alpha_d)
$$
over the complex numbers $\mathbb{C}$. 
The height of $\alpha$, also known as the \textit{absolute Weil height} of $\alpha$ and denoted by $\wH(\alpha)$, is defined by
\begin{equation} 
\label{eq: Weil}
\wH(\alpha)=\(a_d\prod^d_{j=1}\max\{1,|\alpha_j|\}\)^{1/d}.
\end{equation}

The next result shows that if algebraic numbers $\alpha_1,\dots,\alpha_n$ are multiplicatively dependent, then there is a 
dependence relation where the exponents are  not too large in absolute value; 
see for example~\cite[Theorem 3]{Loxton} or~\cite[Theorem~1]{Poorten}.

\begin{lemma} 
\label{lem:exponent}
Let  $n\geq 2$ and let $\alpha_1,\dots,\alpha_n$ be multiplicatively dependent non-zero algebraic numbers of degree at most $d$ which are not roots of unity. Then there is a positive number $c$, which depends only on $n$ and $d$, and there are rational integers $k_1,\dots,k_n$, not all zero, such that
$$
\alpha^{k_1}_1\cdots\alpha^{k_n}_n=1
$$
and
$$
|k_j|\le c\prod_{m=1, \, m \ne j}^{n} \log \wH(\alpha_j), \qquad j=1,\ldots,n.
$$
\end{lemma} 

We remark that the upper bound in Lemma~\ref{lem:exponent} is best possible up to a multiplicative constant; see~\cite[Example~1]{Loxton}. 

The following result describes the typical form of a  two dimensional multiplicatively dependent vector over a number field.

\begin{lemma} \label{lem:mult2}
Let $K$ be a number field, and let $h$ be the class number of $K$. 
If $\alpha$ and $\beta$ in $K$ are multiplicatively dependent, then there exists $\gamma$ in $K$ such that 
$(\alpha^h,\beta^h)=(\eta_1\gamma^l,\eta_2\gamma^m)$ for roots of unity $\eta_1,\eta_2$ from $K$ and some integers $l$ and $m$. 
\end{lemma}

\begin{proof}
Since $\alpha$ and $\beta$ are multiplicatively dependent, without loss of generality we can assume that there exist two positive integers $k_1,k_2$ such that 
\begin{equation} 
\label{eq:albe}
\alpha^{k_1} = \beta^{k_2}. 
\end{equation}
First, we look at the prime decompositions of the fractional ideals $\langle \alpha \rangle$ and $\langle \beta \rangle$ of $K$. 
Notice that there exist distinct prime ideals $\p_1,\ldots, \p_n$ of $K$ and integers $e_1,\ldots,e_n,s_1,\ldots,s_n$ such that 
$$
\langle \alpha \rangle = \p_1^{e_1} \ldots \p_n^{e_n}, \qquad 
\langle \beta \rangle = \p_1^{s_1} \ldots \p_n^{s_n}, 
$$
which, together with~\eqref{eq:albe}, implies that 
\begin{equation} \label{eq:krs}
 k_1e_j = k_2s_j,  \qquad j = 1, \ldots, n.
\end{equation}
Then, choosing integers
$$
l = \frac{k_2}{\gcd(k_1,k_2)}, \quad m = \frac{k_1}{\gcd(k_1,k_2)}
$$
and 
$$
t_j = \frac{e_j \cdot \gcd(k_1,k_2)}{k_2} = \frac{s_j \cdot \gcd(k_1,k_2)}{k_1}, \quad j = 1, \ldots, n, 
$$
we have 
$$
\langle \alpha \rangle = (\p_1^{t_1} \ldots \p_n^{t_n})^l, \qquad 
\langle \beta \rangle = (\p_1^{t_1} \ldots \p_n^{t_n})^m, 
$$
and 
\begin{equation} \label{eq:klm}
k_1 l = k_2 m.
\end{equation}
Since $h$ is the class number of $K$, the fractional ideal $(\p_1^{t_1} \ldots \p_n^{t_n})^h$ is principal. 
That is, there exists an element $\gamma_0 \in K$ such that 
$$(\p_1^{t_1} \ldots \p_n^{t_n})^h=\langle \gamma_0 \rangle,
$$ 
and thus 
$$
\langle \alpha^h \rangle = \langle \gamma_0^l \rangle, \qquad 
\langle \beta^h \rangle = \langle \gamma_0^m \rangle.
$$
So, there are two units $u,v$ of $\cO_K$ such that 
\begin{equation}  \label{eq:albega0}
\alpha^h = u \gamma_0^l, \qquad  \beta^h = v \gamma_0^m. 
\end{equation}

Now, by~\eqref{eq:albe}, \eqref{eq:klm} and~\eqref{eq:albega0}, we obtain 
\begin{equation} \label{eq:uv}
u^{k_1} = v^{k_2} . 
\end{equation} 
Let $r$ be the rank of the group of units of $\cO_K$. 
By Dirichlet's unit theorem, there exist $r$ fundamental units $w_1,\ldots,w_r\in \cO_K$ such that 
\begin{equation} \label{eq:uvw}
u=\eta_1w_1^{a_1}\ldots w_r^{a_r} \mand v=\eta_2w_1^{b_1}\ldots w_r^{b_r}
\end{equation}
for some roots of unity $\eta_1,\eta_2 \in K$ and  integers $a_1,\ldots,a_r,b_1,\ldots,b_r$. 
Clearly,~\eqref{eq:uvw} also includes the case when the rank  $r=0$.
We substitute~\eqref{eq:uvw} into~\eqref{eq:uv} and deduce that 
$$
\eta_1^{k_1} = \eta_2^{k_2} 
$$
and 
\begin{equation} \label{eq:kab}
 k_1a_j = k_2b_j ,  \qquad j = 1, \ldots, r.
\end{equation} 

By~\eqref{eq:krs} and~\eqref{eq:kab}, there exists a unit $w \in \cO_K$ such that 
$$
u=\eta_1w^l 
\mand 
v=\eta_2w^m, 
$$
where $l$ and $m$ have been defined in the above. 
Substituting this into~\eqref{eq:albega0} and denoting $w\gamma_0$ by $\gamma$   we have  
$$
\alpha^h = \eta_1\gamma^l  
\mand 
\beta^h = \eta_2\gamma^m. 
$$
This completes the proof. 
\end{proof}

\subsection{Gaps between products of powers of fixed primes}
\label{sec:gap} 
We need a  result of  Tijdeman~\cite{Tij1} on a lower bound on the gaps between integers of the form 
$p_1^{s_1}\cdots p_k^{s_k}$ for distinct primes $p_1, \ldots, p_k$ and non-negative
integers $s_1, \ldots, s_k$, $k \ge 2$, see also~\cite{Lan,Tij2}.

\begin{lemma} \label{lem:Gap large} 
Let $S=\{p_1, \ldots, p_k\}$ be a nonempty set of prime numbers and let $m_1 < m_2 <  \ldots$ be the increasing sequence of positive integers  composed of primes from $S$. Then there exists a positve number $c$ which is effectively computable in terms of $S$ such that 
$$
m_{j+1} - m_j \gg  \frac{m_j}{(\log m_j)^{c}} . 
$$
\end{lemma}

\section{Proofs of Density Results}

\subsection{Proof of Theorem~\ref{thm:Dense_S_R}}
We first note that it suffices to prove our result for $n=2$, as then we can approximate 
any vector $(x_1,\ldots, x_n)$ by $(v_1,v_2, v_3, \ldots, v_n)\in S^n$,  where 
\begin{itemize}
\item $v_1,v_2$  are multiplicatively dependent and chosen  to approximate $x_1,x_2$ respectively, 
\item $ v_3, \ldots, v_n$ are chosen independently to approximate $x_3, \ldots, x_n$.  
\end{itemize}

Let $(x_1,x_2) \in \R^2$. 
It is enough to prove that for any $\varepsilon > 0$ there exists an element 
of $\cM_2(S)$ which differs from $(x_1,x_2)$ by at most $\varepsilon$ in each coordinate. 
For each $\varepsilon > 0$, we choose a real number 
$\delta>0$, depending only on $\varepsilon$, $x_1$ and $x_2$, such that if $\alpha$ is a
real number with 
\begin{equation} 
\label{eq: good alpha}
-1 -  \delta < \alpha < -1 -  \delta/2, 
\end{equation} 
there exist integers $k$ and $m$ such that 
\begin{equation} 
\label{eq: alpha approx}
\begin{split}
& ||\alpha|^k - |x_1|| <\varepsilon, \quad ||\alpha|^{k+1} - |x_1|| < \varepsilon,  \\
& ||\alpha|^m - |x_2|| <\varepsilon, \quad  ||\alpha|^{m+1} - |x_2|| <\varepsilon . 
\end{split}
\end{equation} 
Since $S$ is dense in $\R$, there is an element $\alpha$ in $S$ satisfying~\eqref{eq: good alpha}.  
Since $S$ is closed under powering, 
we have that the four vectors $(\alpha^k, \alpha^m)$, $(\alpha^k, \alpha^{m+1})$, $(\alpha^{k+1}, \alpha^m)$ 
and $(\alpha^{k+1}, \alpha^{m+1})$ are all in $\cM_2(S)$ and also~\eqref{eq: alpha approx} holds. 
Note that at least one of these four vectors differs from $(x_1,x_2)$ by at most $\varepsilon$ in each coordinate 
(according to the signs of $x_1$ and $x_2$).  
The desired result now follows.

\subsection{Proof of Theorem~\ref{thm:Dense_S_C}}

As in the proof of Theorem~\ref {thm:Dense_S_R} we observe that it suffices to prove our result for $n=2$.

Let $(z_1,z_2) \in \C^2$. We show that there is a sequence of elements 
of  $\cM_2(S)$ which converges to $(z_1, z_2)$.

 We  first prove the result when $z_1z_2=0.$
Without loss of generality we may suppose that $z_1 = 0$. 
If $|z_2| < 1$,  we let $(s_1,s_2,\ldots)$
be a sequence of complex numbers from $S$ with $|s_m| \leq |z_2|$ for $m= 1,2, \ldots $ which converges to $z_2$. 
Then $(s_m^m,s_m)$ is in $\cM_2(S)$ for $m= 1,2, \dots$ and
$$
  \lim_{m \to \infty} (s_m^m,s_m) =(0,z_2).
$$
 On the other hand, if $|z_2| \geq 1$, let $(s_1,s_2,\ldots)$
be a sequence of complex numbers from $S$ with $|s_m| \geq (1+\frac{1}{m})|z_2|$ for $m= 1,2, \ldots $ which converges to $z_2$. 
Then $(s_m^{-m^2},s_m)$ is in $\cM_2(S)$ for $m=1,2, \ldots$, and since $|z_2| \geq 1$,
$$
 \lim_{m \to \infty} (s_m^{-m^2},s_m) =(0,z_2).
   $$

We now suppose that $z_1z_2 \ne 0$. Put 
$$
z_j = |z_j| \exp(2 \pi i \vartheta_j), 
$$
where $i = \sqrt{-1}$ and $0 \le \vartheta_j < 1$, for $j=1,2$. For each positive integer $m$ we put
$$
\omega_m = \(1+\frac{1}{m^2}\) \exp(2 \pi i/m).
$$
Next, let $a_{j,m}$ be the unique integer with 
$$
\(1+\frac{1}{m^2}\)^{a_{j,m}}  \le |z_j| < \(1+\frac{1}{m^2}\)^{a_{j,m}+1}, \quad j =1,2.
$$  
Define $r_{j,m}\in \{0, \ldots, m-1\}$ by the condition $r_{j,m} \equiv a_{j,m} \pmod m$, 
and then choose $b_{j,m} \in \{-r_{j,m} , -r_{j,m}+1,  \ldots, m-r_{j,m}-1\}$ such that 
$$
\frac{b_{j,m} +r_{j,m}}{m} \le \vartheta_j < \frac{b_{j,m} +r_{j,m}+1}{m} , \quad j =1,2.
$$
Observe that $|b_{j,m}|$ is at most $m$, for $j =1,2$. Then 
$$
\lim_{m \to \infty} \omega_m^{a_{j,m} + b_{j,m}} = z_j, \quad j =1,2.
$$
Since $S$ is dense in $\C$, there is an element $t_m$ in $S$ with$$
\left|t_m^{a_{j,m} +b_{j,m}} -  \omega_m^{a_{j,m} +b_{j,m}} \right| < \frac{1}{m},   \quad j =1,2, \ 
 m =1,2,\ldots . 
$$
Thus 
$$
\lim_{m \to \infty} (t_m^{a_{1,m} +b_{1,m}}, t_m^{a_{2,m} +b_{2,m}} )= (z_1,z_2). 
$$
 Since $S$ is closed under powering, 
we see that each $(t_m^{a_{1,m} +b_{1,m}}, t_m^{a_{2,m} +b_{2,m}} )$ is in $\cM_2(S)$, and the result now follows.

\section{Proofs of Bounds on the Covering Radius}

 \subsection{Proof of the lower bound of Theorem~\ref{thm:rho H Z}}
 
 We start with $n=2$. Since the upper bound is trivial, we only need to prove the lower bound. 
Fix a vector $\xb=(H/2,3H/4)$, we have $\| \xb \| \le H$. 
For any vector $\vb=(v_1,v_2)\in \cM_2(\Z)$ with $v_1 > 1$ and $v_2 > 1$, by Lemma~\ref{lem:mult2} we can
  choose a positive integer $a\ge 2$ such that $v_1=a^{s_1}, v_2=a^{s_2}$ for some positive integers $s_1,s_2$. 

Assume that $s_1 \ge s_2$, that is $v_1\ge v_2$.  If $v_1\le 5H/8$, then $\| \xb - \vb \|\ge 3H/4-5H/8=H/8$. 
Otherwise if $v_1 > 5H/8$, we have $\| \xb - \vb \|\ge 5H/8 - H/2=H/8$. So, in this case we obtain 
$$
\| \xb -\vb \| \ge H/8. 
$$

Now, we assume that $s_1 < s_2$. 
If $v_1\le 5H/12$, then $\| \xb - \vb \|\ge H/2-5H/12=H/12$;  
while if $v_1 > 5H/12$, then 
$$
v_2 \ge a v_1 \ge 2 v_1 >5H/6, 
$$
which implies that $\|\xb-\vb\| > 5H/6 - 3H/4 =H/12$. Hence, in this case we have 
$$
\| \xb -\vb \| \ge H/12
$$
as required. 
 
 We now consider the case $n\ge 3$. 
Let $p_1, \ldots, p_n$ be the
first $n$ primes.  We define $q_j$ as the largest power of $p_j$ 
which does not exceed $H/2$; thus we  have 
$$
\frac{H}{2p_j} < q_j \le \frac{H}{2}, \quad j=1, \ldots, n. 
$$
We now set $b=(c+2)n$, where $c$ is the constant in Lemma~\ref{lem:Gap large}
which corresponds to $k=n$ and the above choice of primes. 

We now define the $n$-dimensional box 
\begin{align*}
\fB &= \left[q_1-H/(\log H)^b, q_1+H/(\log H)^b\right] 
\times \\
& \qquad \qquad \ldots  \times  \left[q_n-H/(\log H)^b, q_n+H/(\log H)^b\right], 
\end{align*}
and show that $\cM_n(\Z) \cap \fB =\emptyset$ when $H$ is sufficiently large. 
Indeed, we assume that  there exists $\vb = (v_1,\ldots, v_n) \in \cM_n(\Z) \cap \fB$. 
Then, by Lemma~\ref{lem:exponent} one can choose the exponents 
$k_j$ in~\eqref{eq:0} to satisfy
\begin{equation}
\label{eq:small k_j}
|k_j| \ll (\log H)^{n-1}, \qquad j =1, \ldots, n.
\end{equation}
Since
$$
v_j = q_j + O(H/(\log H)^b) =  q_j\( 1 + O(1/(\log H)^b)\), 
$$
using~\eqref{eq:small k_j} we also have 
$$
v_j^{k_j} =   q_j^{k_j} \(1 + O\(|k_j|/(\log H)^b\)\)
=   q_j^{k_j} \(1 + O\(1/(\log H)^{b-n+1}\)\)
$$
for $j =1, \ldots, n$. 
Hence 
$$
\prod_{j=1}^n  q_j^{k_j} = 1 + O\(1/(\log H)^{b-n+1}\).
$$
Collecting negative and positive exponents we rewrite this as 
\begin{equation}
\label{eq:Q+Q-}
Q_+ = Q_-\(1 + O\(1/(\log H)^{b-n+1}\)\),  
\end{equation}
where
$$
Q_+ = \prod_{\substack{j=1\\k_j > 0}}^n  q_j^{k_j}  \mand
Q_- = \prod_{\substack{j=1\\k_j < 0}}^n  q_j^{-k_j}.
$$
Since by~\eqref{eq:small k_j} we have 
$$
\max \{ \log Q_+, \log Q_- \}  \ll (\log H)^{n}, 
$$
we can rewrite~\eqref{eq:Q+Q-} as 
\begin{equation}
\label{eq:too close}
Q_+ = Q_-\(1 + O\(1/(\log Q_*)^{(b-n+1)/n}\)\)
\end{equation}
where $Q_* = \min\{Q_+ , Q_-\}$.
Since due to our choice of $b$ we have 
$$
(b-n+1)/n = (b+1)/n -1 > c+1 > c, 
$$
we see that~\eqref{eq:too close} contradicts Lemma~\ref{lem:Gap large} when $H$ is sufficiently large.
This in fact completes the proof of the lower bound.

 \subsection{Proof of the upper bound of Theorem~\ref{thm:rho H Z}}
We only need to consider the case $n \ge 3$. 
Let 
$$
\bx = (x_1, \ldots,x_n) \in \Z^n 
$$
with $\|\bx \| \le H$ for large enough $H$. 
Without loss of generality we can assume that
$$
0 < x_1 \le  \ldots \le x_n \le H.
$$
Moreover, if $x_1\le H/(\log H)^{n-2}$, then we take $\bu = (1,x_2,\dots, x_n)$
and get
$$
\|\bx - \bu\| \le H/(\log H)^{n-2}.
$$

Hence we can assume that 
\begin{equation}
\label{eq:large x_1 Z}
H/(\log H)^{n-2} < x_1 \le \ldots \le x_n \le H.
\end{equation}

In addition, the case when $x_1=\dots=x_n$ is trivial, and so we assume that
\begin{equation}
\label{eq:neq}
x_n > x_1.
\end{equation}  

Denote 
$$y_i = \log(x_i/x_1), \qquad  i =2, \ldots, n.
$$
By (\ref{eq:neq}), we have $y_n>0$.
Next, we see from~\eqref{eq:large x_1 Z} that for a sufficiently large $H$ we have
\begin{equation}
\label{eq:small yi}
y_i \ll  \log \log H, \qquad  i =2, \ldots, n.
\end{equation}

We set 
\begin{equation}
\label{eq:def q}
q =  \fl{\frac{\log H}{ n^2 \log \log H}} 
\end{equation}
and, using a $q$-ary expansion of $y_i/y_n \in [0,1]$, choose non-negative integers
$a_{i,j}< q$ where $i=2,\dots,n$, $j=1,\dots,n-2$   such that
\begin{equation}
\label{eq:approx ai/q}
\left|\frac{y_i}{y_n} - \sum_{j=1}^{n-2}\frac{a_{i,j}}{q^j}\right| \le \frac 1{q^{n-2}}, 
\qquad  i =2, \ldots, n.
\end{equation}
Take 
\begin{equation}
\label{eq:def k Z}
k= \fl{(\log H)^{n-1}}
\end{equation}
 and choose positive integers $m_1,\dots,m_{n-2}$ so that
\begin{equation}
\label{eq:approx m/k}
\left|\log(m_j/k) -\frac{y_n}{q^j}\right|\le 1/k, \qquad j=1,\dots,n-2.
\end{equation}
Indeed, this can be done because for $ m \ge k$ we have
$$\log((m+1)/k) - \log(m/k)  = \log (1+1/m) < 1/m \le 1/k.$$

It now follows from~\eqref{eq:approx m/k} that 
\begin{equation}
\begin{split} 
\label{eq:approx ai m/k}
\left| \sum_{j=1}^{n-2} a_{i,j}\log(m_j/k)  - \sum_{j=1}^{n-2} \frac{a_{i,j}}{q^j}  y_n\right| &\\
\le \frac{1}{k} \sum_{j=1}^{n-2}a_{i,j} &< qn/k,  \qquad  i =2, \ldots, n.
\end{split} 
\end{equation}

Furthermore,  the inequalities~\eqref{eq:small yi} and~\eqref{eq:approx ai/q} imply  that 
$$
\left|  \sum_{j=1}^{n-2} \frac{a_{i,j}}{q^j}  y_n- y_i\right| 
= y_n\left|\frac{y_i}{y_n} -  \sum_{j=1}^{n-2} \frac{a_{i,j}}{q^j} \right|
\ll \frac {\log\log H}{q^{n-2}},   \qquad  i =2, \ldots, n.
$$
Combining this bound with~\eqref{eq:approx ai m/k},  we obtain
\begin{equation}
\label{eq:approx yi m/k}
\left|\sum_{j=1}^{n-2} a_{i,j}\log(m_j/k)  - y_i\right| \ll \frac {\log\log H}{q^{n-2}} +  \frac{qn}{k},   
\qquad  i =2, \ldots, n.
\end{equation}

Clearly there exists  $u_1=k^{qn}v$ with a positive integer $v$ satisfying
\begin{equation}
\label{eq:def v}
\left|x_1 - u_1\right|  \le k^{qn} .
\end{equation}
Take $\bu =(u_1,\ldots,u_n)$, where 
$$u_i = u_1   \prod_{j=1}^{n-2}(m_j/k)^{a_{i,j}} =   v k^{s_i} \prod_{j=1}^{n-2}m_j^{a_{i,j}}  
 \in \N
$$
with 
$$
s_i = qn - \sum_{j=1}^{n-2}a_{i,j}>0,    
$$
for $i =2, \ldots, n.$ Clearly, $u_1, \ldots, u_n$ generate a multiplicative subgroup of rank 
at most $n-1$ in $\Q^*$ (which is contained in the multiplicative subgroup   generated by $ u_1, m_1/k, \ldots, m_{n-2}/k$)
 and thus $\bu\in \cM_n(\Z)$. 

We now note that the inequalities~\eqref{eq:large x_1 Z} and~\eqref{eq:def v}  imply 
\begin{equation}
\label{eq:approx x1}
\left|\log u_1  - \log x_1\right| = \left|\log (u_1/x_1) \right| \ll \frac{\left| x_1  -  u_1\right|}{x_1}. 
\ll \frac{k^{qn} (\log H)^{n-2}}{H}. 
\end{equation}

Furthermore, for $ i =2, \ldots, n$, using $\log x_i = \log x_1 + y_i$, we see  that 
\begin{align*}
\left|\log u_i - \log x_i\right| & =\left| \sum_{j=1}^{n-2} a_{i,j}\log(m_j/k) 
+ \log u_1 -  \log x_1 - y_i \right| \\
& \le \left| \sum_{j=1}^{n-2} a_{i,j} \log(m_j/k)  - y_i \right|  
+ \left|  \log u_1 -  \log x_1 \right| ,
\end{align*} 
and  from~\eqref{eq:approx yi m/k} and~\eqref{eq:approx x1} we derive 
\begin{equation}
\label{eq:approx xi}
\left|\log u_i  - \log x_i\right| \ll \frac {\log\log H}{q^{n-2}} +  \frac{qn}{k} 
+  \frac{k^{qn} (\log H)^{n-2}}{H}. 
\end{equation}

 One easily checks  that for  the 
choice of parameters~\eqref{eq:def q} and~\eqref{eq:def k Z} we have 
\begin{align*}
& \frac {\log\log H}{q^{n-2}}  \ll  \frac {(\log\log H)^{n-1}}{(\log H)^{n-2}},\\
& \frac{qn}{k}  \ll \frac {\log\log H}{(\log H)^{n-2}}, \\
 &  \frac{k^{qn} (\log H)^{n-2}}{H} \ll     \frac{\(\(\log H\)^{(n-1)n}\)^{\log H/n^2 \log \log H} (\log H)^{n-2}}{H} \\
& \qquad \qquad \qquad =   \frac{H^{(n-1)/n} (\log H)^{n-2}}{H}
 =    \frac{(\log H)^{n-2}}{H^{1/n}}.
  \end{align*} 
  Clearly the bound~\eqref{eq:approx xi}  absorbs~\eqref{eq:approx x1},  
thus for $ i =1, \ldots, n$ we have
$$
\left|\log u_i  - \log x_i\right| \ll \frac {(\log\log H)^{n-1}}{(\log H)^{n-2}}, 
$$
which implies 
\begin{align*}
u_i  & = x_i\exp\(O\( (\log\log H)^{n-1}/(\log H)^{n-2}\)\)  \\
& = x_i +O\( x_i(\log\log H)^{n-1}/(\log H)^{n-2}\).
  \end{align*} 
Hence,
$$
\|\bx - \bu\|  \ll  H \frac {(\log\log H)^{n-1}}{(\log H)^{n-2}},
$$
which concludes the proof of the upper bound.

 \subsection{Proof of the lower bound of  Theorem~\ref{thm:ImQuad}}
We start with $n=2$. 
Let $h$ be the class number of $K$. 
We first fix a  number $c$ with 
$$
0 < c < 2^{-(h+1)/(2h)}, 
$$
and then fix a vector 
$$
\zb=(z_1,z_2)=(aH,bH) 
$$ 
and another real number $d$ such that  
\begin{equation}   \label{eq:abcd}
0<c< a <d< b < 2^{1/(2h)}c. 
\end{equation}
It is easy to see that $\| \zb \| \le H$.

For any vector $\vb=(v_1,v_2)\in \cM_2(\cO_K)$ with $|v_1| > 1$ and $|v_2| > 1$, by Lemma~\ref{lem:mult2} we can
  choose an element $\gamma \in \cO_K$ with $|\gamma| > 1$ such that $v_1=\eta_1\gamma^{s_1/h}, v_2=\eta_2\gamma^{s_2/h}$ for 
 some roots of unity $\eta_1,\eta_2$ and some positive integers $s_1,s_2$. 
 Clearly, we have 
 \begin{equation} \label{eq:z-v}
 \begin{split}
 \| \zb - \vb \| & = \| (z_1-v_1,z_2-v_2)\| \ge \max\{|z_1-v_1|, |z_2-v_2| \}  \\
 & \ge \max\{\big| |z_1|-|v_1| \big|, \big| |z_2|-|v_2| \big| \}. 
 \end{split}
 \end{equation}
 Besides, since $\gamma \in \cO_K$ with $|\gamma|>1$ and $K$ is an imaginary quadratic field, we have 
 $$
 |\gamma| \ge \sqrt{2}.
 $$

Assume that $s_1 \ge s_2$, that is $|v_1| \ge |v_2|$.   If $|v_1|\le dH$, then 
$$
\big| |z_2| - |v_2| \big|\ge bH-dH=(b-d)H.
$$
 Otherwise, if $|v_1| > dH$, we have 
 $$
 \big| |z_1| - |v_1| \big|\ge dH - aH=(d-a)H.
 $$
  So, in this case using~\eqref{eq:z-v} we obtain 
\begin{equation}  \label{eq:z-v1}
\| \zb -\vb \| \ge \min\{(b-d)H, (d-a)H\}. 
\end{equation}

Now, we assume that $s_1 < s_2$. 
If $|v_1|\le cH$, then 
$$
\big| |z_1| - |v_1| \big| \ge aH-cH=(a-c)H;
$$  
while if $|v_1| > cH$, then 
$$
|v_2| \ge |\gamma|^{1/h} |v_1| \ge 2^{1/(2h)} |v_1| > 2^{1/(2h)}cH, 
$$
which implies that 
$$
\big| |z_2|- |v_2| \big| > 2^{1/(2h)}cH - bH = (2^{1/(2h)}c-b)H.
$$
 Hence, in this case using~\eqref{eq:z-v} we have 
\begin{equation}  \label{eq:z-v2}
\| \zb -\vb \| \ge \min\{ (a-c)H,  (2^{1/(2h)}c-b)H\}. 
\end{equation}
Combining~\eqref{eq:z-v1}, \eqref{eq:z-v2} with~\eqref{eq:abcd}, 
we conclude the proof for the case $n=2$. 

For the case $n \ge 3$, we recall the box $\fB$ defined in the proof of Theorem~\ref{thm:rho H Z}. 
Applying the same arguments as before, we obtain that  for sufficiently large $H$, 
$$
\{(|v_1|,\ldots,|v_n|): \, (v_1,\ldots,v_n) \in \cM_n(\cO_K)\} \cap \fB = \emptyset, 
$$
where we also need to use the fact that for any $\alpha \in \cO_K$, if $|\alpha| \le H$, 
then for its height we have $\log \wH(\alpha) \ll \log H$, and so $\log \wH(|\alpha|) \ll \log H$ (because $K$ is an imaginary quadratic field).  
This gives the desired lower bound.

 \subsection{Proof of the upper bound of  Theorem~\ref{thm:ImQuad}}
We only need to prove the upper bound for $n \ge 3$.   In fact it is sufficient 
to consider only the case $n=3$, since this automatically means that for any 
$n \ge 3$ and $\bx\in\cO_K^n$,   there is
a vector $\bu\in \cM_n(\cO_K)$ such that
the distance $\|\bx - \bu\|$ is small.  

Let $\bx = (x_1, x_2, x_3) \in \cO_K^3$ with $\|\bx \| \le H$ for large enough $H$. 
We want to show that there is
a vector $\bu\in \cM_3(\cO_K)$ such that
the distance $\|\bx - \bu\|$ is small.

Without loss of generality we can assume that
$$
0 < |x_1| \le  |x_2| \le |x_3| \le H.
$$
Moreover, if $|x_1|\le H/\log H$, then we take $\bu = (1,x_2,x_3)$
and get
$$
\|\bx - \bu\| \le H/\log H.
$$

Hence we can assume that
\begin{equation}
\label{eq:large x_1 OK}
H/\log H < |x_1| \le |x_2| \le |x_3| \le H.
\end{equation}

Denote
$$
q=\fl{\sqrt{\log H}/(2\log\log H)} \mand 
\lambda = \exp\left(\frac{1}{q^2}  + \frac{2\pi i}q\right), 
$$
where $i = \sqrt{-1}$. 
Our plan is to approximate the numbers $x_2$ and $x_3$ by 
$z_{j_2}$ and $z_{j_3}$, respectively, where $j_2$ and $j_3$
are non-negative integers,  where
$$z_j = \lambda^j x_1,\qquad j\ge 0,$$ 
and then approximate $x_1, z_{j_2}, z_{j_3}$ by multiplicatively dependent elements  
$u_1,u_2,u_3$ of $\cO_K$.

Let
$$\nu_1 = \fl{q\log(|x_2|/|x_1|)}.$$
By~\eqref{eq:large x_1 OK}, we get
\begin{equation}
\label{eq:est_nu1}
0\le \nu_1 \le q\log\log H.
\end{equation}
One verifies that 
$$
| \lambda|^{(\nu_1+1)q}  \le  | \lambda|^{q^2\log(|x_2|/|x_1|) + q} = \exp(1/q) |x_2|/|x_1|
$$
and 
$$
| \lambda|^{\nu_1q}  \ge | \lambda|^{q^2\log(|x_2|/|x_1|)- q}  = |x_2|/ (|x_1|  \exp(-1/q)).
$$
Hence, for every integer $j$ with $\nu_1q \le j < (\nu_1+1)q$,  we have
\begin{equation}
\label{eq:est_modul}
|\log(|x_2|/|\lambda^j x_1|)|\le 1/q.
\end{equation}  

We assume that the principal argument $\arg(z)$ of a non-zero complex number $z$ always belongs to 
the interval  $[0,2\pi)$ and denote
$$
a=\arg(x_2/x_1)\in[0,2\pi).
$$  
Take
$$
\nu_2 =\fl{qa/(2\pi)}.
$$
Clearly,  
\begin{equation}
\label{eq:est_nu2}
0\le\nu_2 <q.
\end{equation}
Then, since $\lambda^q \in \R$,  we have 
\begin{equation}
\label{eq:est_arg}
\begin{split}
 \arg\(x_2/\(\lambda^{\nu_1q+\nu_2}x_1\)\) & = \arg\(x_2/\(\lambda^{\nu_2}x_1\)\) \\
 & = \arg\( \lambda^{\{qa/(2\pi)\}}\)   \in[0,2\pi/q), 
\end{split} 
\end{equation}
where $\{\xi\}$ denotes the fractional part of a real $\xi$.  

Define
$$j_2 = \nu_1q+\nu_2.$$
We conclude from~\eqref{eq:est_nu1} and~\eqref{eq:est_nu2}  that
\begin{equation}
\label{eq:est_j2}
0\le j_2 \le Q,
\end{equation}
where 
\begin{equation}
\label{eq:def Q}
Q= \fl{q^2\log\log H} +q.
\end{equation}
Furthermore, from~\eqref{eq:est_modul} and~\eqref{eq:est_arg} we see  that
\begin{equation}
\label{eq:approx_x2}
 |x_2-\lambda^{j_2}x_1|\ll |x_2|/q.
\end{equation}

Similarly, we can choose $j_3$ satisfying
\begin{equation}
\label{eq:est_j3}
0\le j_3 \le Q
\end{equation}
and
\begin{equation}
\label{eq:approx_x3}
 |x_3-\lambda^{j_3}x_1|\ll |x_3|/q.
\end{equation}

Take
\begin{equation}
\label{eq:def k OK}
k= \fl{q^3\log\log H}.
\end{equation}
Since  $\cO_K$  has a structure of a lattice, we can choose $\beta \in\cO_K$ such that $|\beta - k\lambda| \ll 1$, and so 
\begin{equation}
\label{eq:approx mu/k}
\left|\beta/k -\lambda\right|\ll 1/k \ll q^{-3}(\log\log H)^{-1}. 
\end{equation}

Set
$$\widetilde\lambda = \beta/k.$$
Writing $\zeta =  \widetilde\lambda /\lambda -1$ we observe that 
$$
(1+\zeta)^j = 1 + O(j \zeta)
$$
provided $j \zeta < 1/2$.  
By~\eqref{eq:approx mu/k} 
$$
Q \left|\widetilde\lambda /\lambda - 1 \right| \ll Q/k \ll 1/q, 
$$
and so for any integer  $j$  with $0\le j \le Q$
we have
$$\left|\widetilde\lambda^j -\lambda^j\right| \ll j  |\lambda^j| \left|\widetilde\lambda / \lambda - 1\right|  
\le
Q  |\lambda^j| \left|\widetilde\lambda / \lambda - 1\right| \ll |\lambda^j|/q.$$
In particular, by~\eqref{eq:est_j2} and~\eqref{eq:est_j3} this holds for $j = j_2$ and $j_3$. Now, recalling~\eqref{eq:approx_x2} and~\eqref{eq:approx_x3},  we obtain
\begin{equation}
\label{eq:approx_x23}
|x_2-\widetilde\lambda^{j_2}x_1|\ll |x_2|/q 
\mand 
|x_3-\widetilde\lambda^{j_3}x_1|\ll |x_3|/q.
\end{equation}

Taking into account~\eqref{eq:def Q}  and~\eqref{eq:def k OK} we find that
\begin{align*}&Q  \le \left(\frac14+o(1)\right)\log H/\log\log H,\\
& \log k\le\left(\frac32+o(1)\right)\log\log H,
\end{align*}
and we conclude that
$$k^Q \ll (H/\log H)^{1/2} \ll |x_1|^{1/2}.$$
Then, we can choose $v\in\cO_K$ such that $|x_1 / k^Q - v| \ll 1$, and so 
\begin{equation}
\label{eq:approx_x1}
|x_1 -k^Q v| \ll k^Q \ll |x_1|^{1/2} \ll |x_1|/q.
\end{equation}

We set
\begin{align*}
&u_1 = k^Q v,\\
&u_2 = \widetilde\lambda^{j_2} u_1 = k^{Q-j_2} \beta^{j_2} v,\\
&u_3 = \widetilde\lambda^{j_3} u_1 = k^{Q-j_3} \beta^{j_3} v.
\end{align*}
Clearly, $u_1,u_2,u_3$ are multiplicatively dependent elements of
$\cO_K$ (since they belong to the multiplicative group generated by $u_1$ and $\widetilde\lambda$.  Finally, we conclude from~\eqref{eq:approx_x23} 
and~\eqref{eq:approx_x1} that
$$|x_j - u_j| \ll |x_j|/q, \qquad j=1,2,3,$$
implying
$$
\|\bx - \bu\|  \ll  H \frac {\log\log H}{(\log H)^{1/2}} \quad \textrm{for}  \quad \bu=(u_1, u_2, u_3), 
$$
which concludes the proof.

\section{The hypotheses of Theorems~\ref{thm:Dense_S_R} and~\ref{thm:Dense_S_C} }
\label{sec:rem} 

In this section, we show that in Theorem~\ref{thm:Dense_S_R} and Theorem~\ref{thm:Dense_S_C} the property of $S$ being closed under powering cannot be removed. 

For Theorem~\ref{thm:Dense_S_R} we let $S$ be the set of all rational numbers of the form $p/q$ or $-p/q$  with distinct primes   $p, q$.  
Then by~\cite[Theorem~4]{HS} the set $S$ is dense in $\R$ and we now show   
that  $\cM_n(S)$ is not dense in $\R^n$ for any $n \ge 2$. 

Let $(x_1, \ldots, x_n) \in \cM_n(S)$. Then, there are  integers $k_1, \ldots, k_n$, not all zero,  such  that 
\begin{equation}  \label{eq:xi}
x_1^{k_1} \cdots x_n^{k_n} = 1. 
\end{equation}
Indeed, as a first step we 
show that  there are integers  $k_1, \ldots, k_n  \in \{-1,0,1\}$, not all zero, 
such that 
$$
|x_1^{k_1} \cdots x_n^{k_n}| = 1. 
$$
Note that while it is possible to use  {\it Siegel's Lemma}~\cite[Page~213, Hilfssatz]{Siegel} 
(see also~\cite{BombVaal,Vaal}) to show that there exists a nontrivial solution of \eqref{eq:xi} 
with $k_1, \ldots, k_n$ bounded from above as a function of $n$, which is enough 
for our purpose, we  
give a more direct argument to establish this stronger claim. 

Put  
$$ 
y_i  =  
\begin{cases} 
x_i, & \text{if}\ k_i \ge 0,\\
x_i^{-1}, & \text{if}\  k_i<0, 
\end{cases} \qquad i =1, \ldots, n, 
$$ 
and write 
$$
y_i = \eps_i \frac{p_i}{q_i}
$$
with $\eps_i \in \{-1, 1\}$, $i =1, \ldots, n$.  Let $k_{i_1}, \ldots, k_{i_t}$ be all the non-zero 
integers from~\eqref{eq:xi}. Then
\begin{equation}  \label{eq:yi}
y_{i_1}^{|k_{i_1}|} \cdots y_{i_t}^{|k_{i_t}|} = 1. 
\end{equation}
Observe that by~\eqref{eq:yi} we have 
\begin{equation}  \label{eq:sets}
\{p_{i_1}, \ldots, p_{i_t}\} = \{q_{i_1}, \ldots, q_{i_t}\}. 
\end{equation} 
We claim that there exists distinct integers $j_1, \ldots, j_r$ from the set 
$\{i_1, \ldots, i_t\}$ so that
\begin{equation}  \label{eq:yi=1}
|y_{j_1}  \cdots y_{j_r}| = 1. 
\end{equation} 
To see this, consider the path that starts at $p_{i_1}$ and continues according to the following rules. 
If we are at $p_{i_m}$ we connect $p_{i_m}$ with $q_{i_m}$. Next  $q_{i_m}$ is connected to  $p_{i_s}$, where $s$ is the smallest index with $p_{i_s}=q_{i_m}$. 
This step is always possible by virtue of~\eqref{eq:sets}.
If $p_{i_s}$ has already been traversed by the path, we stop. Observe that this gives us a path which terminates in a cycle and the cycle gives us a solution to~\eqref{eq:yi=1}. 

Let $\alpha_1, \ldots, \alpha_n$ be non-zero real numbers, 
and we also assume that for all $n$-tuples $(\delta_1, \ldots, \delta_n) \ne (0, \ldots, 0)$ with $\delta_i\in \{-1,0,1\}$, $ i =1, \ldots, n$,  we have 
$$
\alpha_1^{\delta_1} \cdots \alpha_n^{\delta_n} \ne \pm 1.
$$
For example, we can choose 
\begin{equation}  \label{example}
\(\alpha_1, \ldots, \alpha_n\) = \(2, 2^3, \ldots , 2^{3^{n-1}}\).
\end{equation}
Notice that there is a positive number $c $ such that 
\begin{equation}  \label{eq:>c}
\left| \alpha_1^{\delta_1} \cdots \alpha_n^{\delta_n}- 1\right | > c
\mand 
\left| \alpha_1^{\delta_1} \cdots \alpha_n^{\delta_n}+1\right | > c
\end{equation} 
for any non-zero $n$-tuple $(\delta_1, \ldots, \delta_n)$ with $\delta_i\in \{-1,0,1\}$, $ i =1, \ldots, n$.  Since every element $(x_1, \ldots, x_n) \in \cM_n(S)$ satisfies an identity of the form~\eqref{eq:yi=1}, it follows from~\eqref{eq:>c} that there is a small ball around 
$(\alpha_1, \ldots, \alpha_n)$ which does not contain any element of $\cM_n(S)$.
As a consequence, we see that $\cM_n(S)$ is not dense in $\R^n$. 

For Theorem~\ref{thm:Dense_S_C} we let $S$ be the set of complex numbers of the form $\zeta p/q$ where $\zeta$ is a root of unity and $p$ and $q$ are distinct primes. Since the roots of unity are dense in the unit circle and the quotients of the primes are dense in the positive real numbers we see that $S$ is dense in $\C$.

We now repeat our argument as before but with
$$
y_i = \zeta_i \frac{p_i}{q_i}
$$
where $\zeta_i$ is a root of unity and $p_i$ and $q_i$ are distinct primes for $i =1, \ldots, n$.
We again find that~\eqref{eq:yi=1} holds. Let $(\alpha_1, \ldots ,\alpha_n)$ be an $n$-tuple of non-zero complex numbers with
$$
|\alpha_1^{\delta_1} \cdots \alpha_n^{\delta_n}| \ne 1
$$
for any non-zero $n$-tuple $(\delta_1, \ldots, \delta_n)$ with 
$$
\delta_i\in \{-1,0,1\}, \qquad i =1, \ldots, n.
$$ 
Plainly~\eqref{example}  gives such an $n$-tuple. Then there is a small ball around $(\alpha_1, \ldots ,\alpha_n)$ which does not contain any element of $\cM_n(S)$ and so $\cM_n(S)$ is not dense in $\C^n$.

\section*{Acknowledgements}
The authors would like to thank the referee for careful reading and valuable comments. 
The first author was able to participate in this project
due to the Nineteenth Annual Workshop on Combinatorial and Additive Number Theory
(CANT 2021, May 24--28, 2021);  he is very grateful
to Melvyn Nathanson who organized this beautiful conference. 
The second author was partly supported by a Macquarie University Research Fellowship 
and by the Australian Research Council Grant DE190100888 and also by the Guangdong Basic and Applied Basic Research Foundation (No. 2022A1515012032). 
The second   and the third authors were supported by the Australian
Research Council Grant DP170100786. 
The research of the fourth author was supported in part by the Canada Research Chairs Program and by Grant A3528 from the Natural Sciences and Engineering Research Council of Canada.

\end{document}